\documentclass[11pt, a4paper, twoside]{amsart}

\usepackage{amsmath}
\usepackage{amssymb}
\usepackage{amsthm}

\usepackage{graphicx}
\usepackage{subcaption}

\usepackage{xcolor}

\usepackage{todonotes, changes}

\allowdisplaybreaks 

\newcommand{\R}{\mathbb{R}}
\newcommand{\Z}{\mathbb{Z}}

\newcommand{\E}{\mathbb{E}}
\newcommand{\Prob}{\mathbb{P}}

\newcommand{\conv}{\mathrm{conv}}

\newcommand{\supp}{\mathrm{supp}}
\newcommand{\aff}{\mathrm{aff}}

\DeclareMathOperator{\vol}{vol}
\newcommand{\voln}{\vol_n}

\newcommand{\pyr}{\mathrm{pyr}}
\newcommand{\centroid}{\mathrm c}

\newcommand{\enorm}[1]{\lVert #1\rVert} 
\newcommand{\ip}[2]{\langle #1,#2\rangle} 
\newcommand{\Pn}{{\mathcal P}^n} 
\newcommand{\wonull}{\setminus\{0\}} 
\newcommand{\inte}{\mathrm{int}\,}  

\newcommand{\V}{\mathrm{V}}
\newcommand{\sph}{\mathbb{S}}
\newcommand{\dual}[1]{{#1}^\star}

\DeclareMathOperator{\lin}{lin}

\usepackage{hyperref}
\hypersetup{backref, pdfpagemode=FullScreen, colorlinks=true,
  citecolor=magenta, linkcolor=cyan, urlcolor=blue}

\title[Affine subspace concentration conditions]{Affine
  subspace concentration conditions for centered polytopes}
\author{Ansgar Freyer, Martin Henk and Christian Kipp}
\address{Technische Universit\"at Berlin, Institut f\"ur Mathematik, Sekr. MA4-1, Stra{\ss}e des 17 Juni 136, D-10623 Berlin}
\address{}
\email{\{freyer, henk, kipp\}@math.tu-berlin.de}
\date{}
\thanks{}

\numberwithin{equation}{section}

\usepackage{mathtools}
\mathtoolsset{showonlyrefs=true}

\begin{document}

\theoremstyle{plain}
\newtheorem{theorem}{Theorem}[section]
\newtheorem{lemma}[theorem]{Lemma}
\newtheorem{corollary}[theorem]{Corollary}
\newtheorem{conjecture}{Conjecture}
\newtheorem{proposition}[theorem]{Proposition}
\newtheorem*{question}{Question}
\newtheorem{thmx}{Theorem}
\renewcommand{\thethmx}{\Alph{thmx}}
\newtheorem{lemmax}[thmx]{Lemma}

\newtheorem{theoman}{Theorem}[section]

\renewcommand{\thetheoman}{\Roman{theoman}}

\theoremstyle{definition}
\newtheorem*{definition}{Definition}

\newtheorem{remark}[theorem]{Remark}
\newtheorem{claim}{Claim}
\newtheorem*{remark*}{Remark}
\newtheorem{example}[theorem]{Example}

\begin{abstract}
  Recently, K.-Y. Wu introduced affine subspace concentration
  conditions for the cone volumes of polytopes and proved that the
  cone volumes of centered, reflexive, smooth
        lattice polytopes satisfy these conditions.  
 We extend the result to arbitrary centered polytopes.
\end{abstract}

\maketitle

\section{Introduction and Results}
\label{sec:intro}
Let $\R^n$ be the $n$-dimensional Euclidean space equipped with the
standard inner product $\ip{\cdot}{\cdot}$ and the Euclidean
norm $\enorm{x}=\sqrt{\ip{x}{x}}$, $x\in\R^n$. Let $\Pn_o$ be the
family of all $n$-dimensional polytopes $P\subset\R^n$ containing the
origin in its interior, i.e., $0\in\inte P$. 
Given such a polytope $P\in \Pn_o$, it admits a unique representation
as
\begin{equation*}
P= \{x\in\R^n: \ip{a_i}{x}\leq 1,\,1\leq i \leq m\},
\end{equation*}   
where the vectors $a_i\in\R^n\wonull$ are pairwise different and $F_i=P\cap\{x\in\R^n
: \ip{a_i}{x}=1\}$, $1\leq i\leq m$, are the facets of $P$. Then the volume of $P$ (i.e., the
$n$-dimensional Lebesgue measure of $P$) can be
written as 
\begin{equation*}
\vol(P)= \frac{1}{n} \sum_{i=1}^m \vol_{n-1}(F_i)\frac{1}{\enorm{a_i}},
\end{equation*}
where, in general, for a $k$-dimensional set $S\subseteq\R^n$, $\vol_k(S)$ denotes
the $k$-dimensional Lebesgue measure with respect to the space $\aff S$, the
affine hull of $S$.   This identity is also known as the pyramid formula,
as it sums up the volumes of the pyramids (cones)
\begin{equation*}
C_i = \conv(\{0\}\cup F_i), 
\end{equation*}
where $\conv S$ denotes the convex hull of the set $S$.     
Observe that     
\begin{equation*}
  \vol(C_i)=\frac{1}{n}\frac{1}{\enorm{a_i}}\vol_{n-1}(F_i),\, 1\leq
  i\leq m.
\end{equation*}
  These cone volumes  are the geometric base of
the  cone-volume measure of an arbitrary convex body, which is  a finite positive Borel measure  on the
$(n-1)$-dimensional unit sphere $\sph^{n-1}\subset\R^n$. The cone-volume measure is the subject of  the
well-known and important 
log-Minkowski problem in modern Convex Geometry, see, e.g.,
\cite{ BoeroeczkyHegedues2015,
  BoeroeczkyHegedusZhu2016, bh16, BoeroeczkyHenkPollehn2018,
  blyz13, BoeroeczkyLutwakYangEtAl2019, clz19, hl14,
  HuangLutwakYangEtAl2016, HugLutwakYangEtAl2005}. 

In the discrete
setting, i.e., the polytopal case, the cone-volume measure $\V_P(\cdot)$
associated to $P$ is the discrete measure
\begin{equation*}
\V_P(\eta)= \sum_{i=1}^m \vol(C_i)\,\delta_{u_i}(\eta),   
\end{equation*}   
where $\eta\subseteq \sph^{n-1}$ is a Borel set, and $\delta_{u_i}(\cdot)$ denotes the delta
measure concentrated on $u_i$. In analogy to the classical
Minkowski problem, the discrete log-Minkowski problem asks for
sufficient and necessary conditions such that a discrete Borel measure
$\mu=\sum_{i=1}^m \gamma_i\,\delta_{u_i}(\cdot)$, $\gamma_i\in\R_{>0}$,
$u_i\in\sph^{n-1}$, is the cone-volume measure of a polytope.

B\"or\"oczky, Lutwak, Yang and Zhang settled the general (i.e., not
necessarily discrete) log-Minkowski problem for arbitrary
finite even Borel measures. Here \textit{even} means that $\mu(A)=\mu(-A)$ holds for all Borel sets $A \subseteq \sph^{n-1}$. This assumption corresponds to the case of origin-symmetric
convex bodies; reduced to the discrete setting their result
may be stated as follows: 
\begin{theoman}[B\"or\"oczky-Lutwak-Yang-Zhang, \cite{blyz13}] A discrete even Borel measure
	$\mu:\sph^{n-1}\to\R_{\geq 0}$ given by $\mu=\sum_{i=1}^m
	\gamma_i\,\delta_{u_i}$, $\gamma_i\in\R_{>0}$,
	$u_i\in\sph^{n-1}$, 
	is the cone-volume measure of an origin-symmetric polytope
	$P\in \Pn_o$ if and
	only if the {\em subspace
		concentration condition} is fulfilled, i.e., i) for every  linear
	subspace $L\subseteq\R^n$ it holds 
	\begin{equation}
	\mu(L\cap \sph^{n-1}) =\sum_{i:\,u_i\in L}\gamma_i \leq \frac{\dim
		L}{n} \sum_{i=1}^m \gamma_i = \frac{\dim
		L}{n}\mu(\sph^{n-1}).
	\label{eq:scc}         
      \end{equation}
      and ii),   
        equality holds in \eqref{eq:scc} for a subspace $L$ if and
	only if there exists a complementary  subspace $L'$ such that $\mu$ is
	concentrated on $L\cup L'$.         
\end{theoman}   

In the non-even case, even in the discrete setting, a complete
characterization is still missing, see \cite{clz19} for the state of
the art.  The main problem here is to find the right  
position of the origin.





A polytope $P\in\Pn_o$ is called \emph{centered} if its centroid $\centroid(P)$ is
at the origin, i.e., 
\begin{equation*}
\centroid(P) = \vol(P)^{-1}\int_P x\,\mathrm d x=0.
\end{equation*} 
It is known that centered polytopes satisfy the subspace
concentration condition. 
\begin{theoman}[Henk-Linke, {\cite{hl14}}]
	\label{thm:henklinke}
	Let $P=\{x\in\R^n: \langle a_i,x\rangle\leq 1,\,1\leq i \leq
        m\}$ be a
	centered polytope and let $L\subseteq\R^n$ be a  linear
	subspace. Then, \eqref{eq:scc} holds true, i.e.,
	\begin{equation*}
	\sum_{i:\,a_i\in L}\vol(C_i) \leq \frac{\dim L}{n}\vol( P).
	\end{equation*} 
	Equality is obtained if and only if there exists a complementary
	linear subspace $L'\subseteq\R^n$ to $L$ such that $\{a_i:1\leq i
	\leq m\}\subseteq L \cup L'$.
\end{theoman}
For a generalization to centered convex bodies we refer to \cite{bh16}.

\begin{figure}[hbt]
\centering
\includegraphics[width=.65\textwidth]{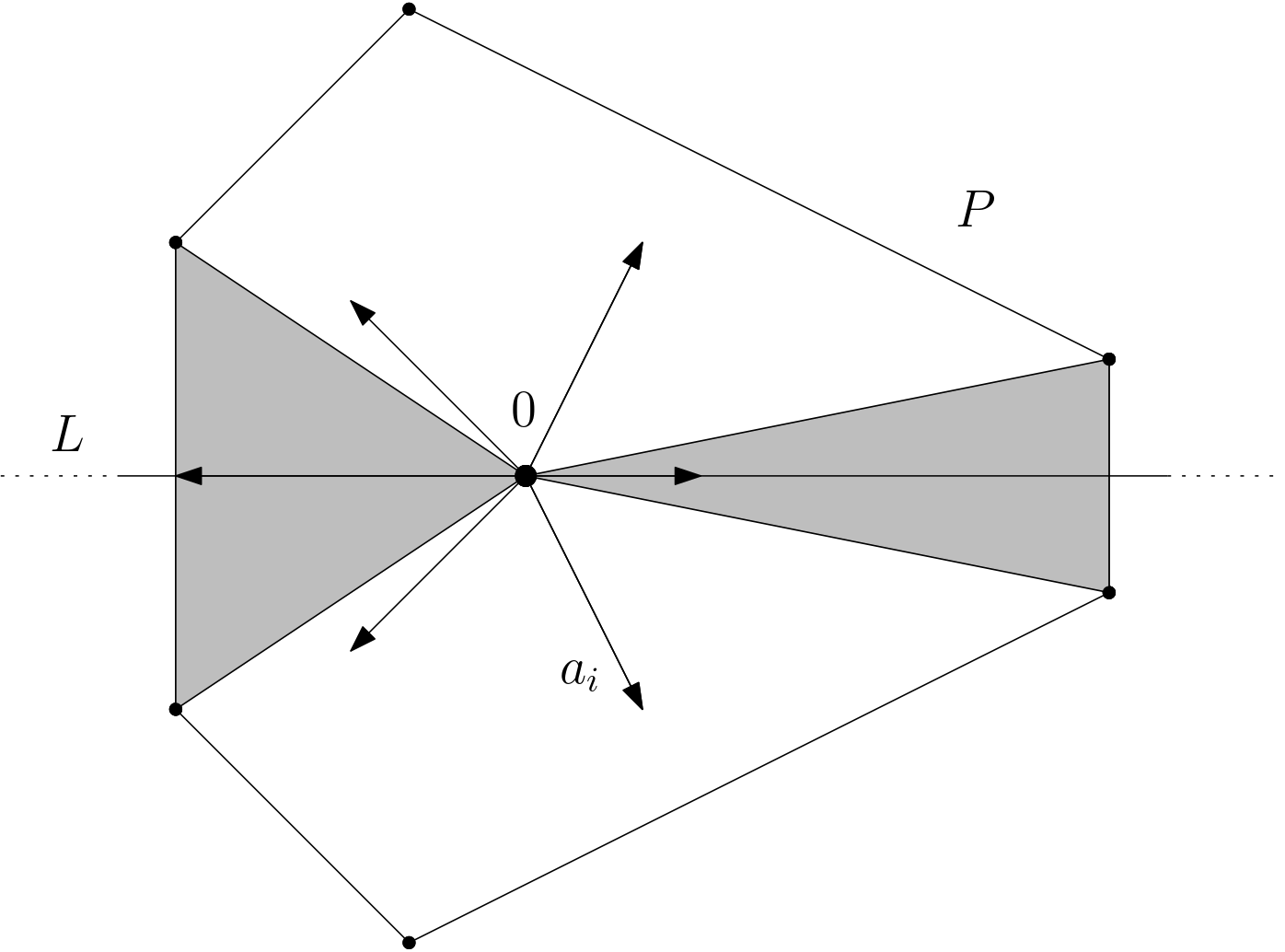}
\caption{The linear situation in Theorem \ref{thm:henklinke}.}
\label{fig:linear}
\end{figure}

Apart from the study of the log-Minkowski problem, the subspace
concentration inequalities have been recently 
reinterpreted in the context of toric geometry in \cite{hns19},
exploiting the deep connection between lattice polytopes and toric
varieties. This lead K.-Y.~Wu to prove an elegant variant to Theorem
\ref{thm:henklinke} in which the linear subspace concentration
condition \eqref{eq:scc} is replaced by an affine subspace
concentration condition. 
\begin{theoman}[K.-Y.~Wu, {\cite{wu22}}]
	\label{thm:wu}
	Let $P=\{x\in\R^n:\langle a_i,x\rangle\leq 1,\,1\leq i\leq
        m\}$ be a
	centered reflexive smooth polytope and let $A\subset\R^n$ be a proper affine subspace. Then, 
	\begin{equation*}
	\sum_{i:\, a_i\in A} \vol(C_i)\leq \frac{\dim A+1}{n+1}\vol(P).
	\end{equation*}
	Equality is obtained if and only if there exists a complementary  affine
	subspace $A'$, i.e., $A\cap A'=\emptyset$ and  $\aff(A\cup
A')=\R^n$, such that $\{a_i:1\leq i \leq m\}\subseteq A \cup A'$. 
\end{theoman}
Here a polytope $P$ is  reflexive if the vectors $a_i$, $1\leq i\leq
m$, as well as the vertices of $P$ are points of $\Z^n$. In other words, $P$ and $\dual{P}$, the polar of $P$, are both lattice polytopes. 
A lattice polytope $P$ is said to be smooth if 
it is simple, i.e., each vertex  of
$P$ is contained in exactly $n$ facets $F_{j_1}, \dots, F_{j_n}$, say,
and 
the corresponding normals $a_{j_1},\dots,a_{j_n}$ form a lattice
basis of $\Z^n$, i.e., $(a_{j_1},\dots,a_{j_n})\Z^n=\Z^n$.

The purpose of this paper is to generalize K.-Y.~Wu's affine subspace
concentration inequalities to arbitrary centered polytopes. 

\begin{theorem}
	\label{thm:affine_scc}
	Let $P = \{x\in\R^n : \langle x,a_i\rangle \leq 1,\,1\leq
        i\leq m\}$
	be a centered polytope and let $A\subseteq\R^n$ be an affine subspace. Then, 
	\begin{equation}
	\label{eq:affine_scc}
	\sum_{i:\,a_i\in A} \vol(C_i)\leq \frac{\dim A+1}{n+1}\vol(P).
	\end{equation}
\begin{figure}[hbt]
\centering
\includegraphics[width=.6\textwidth]{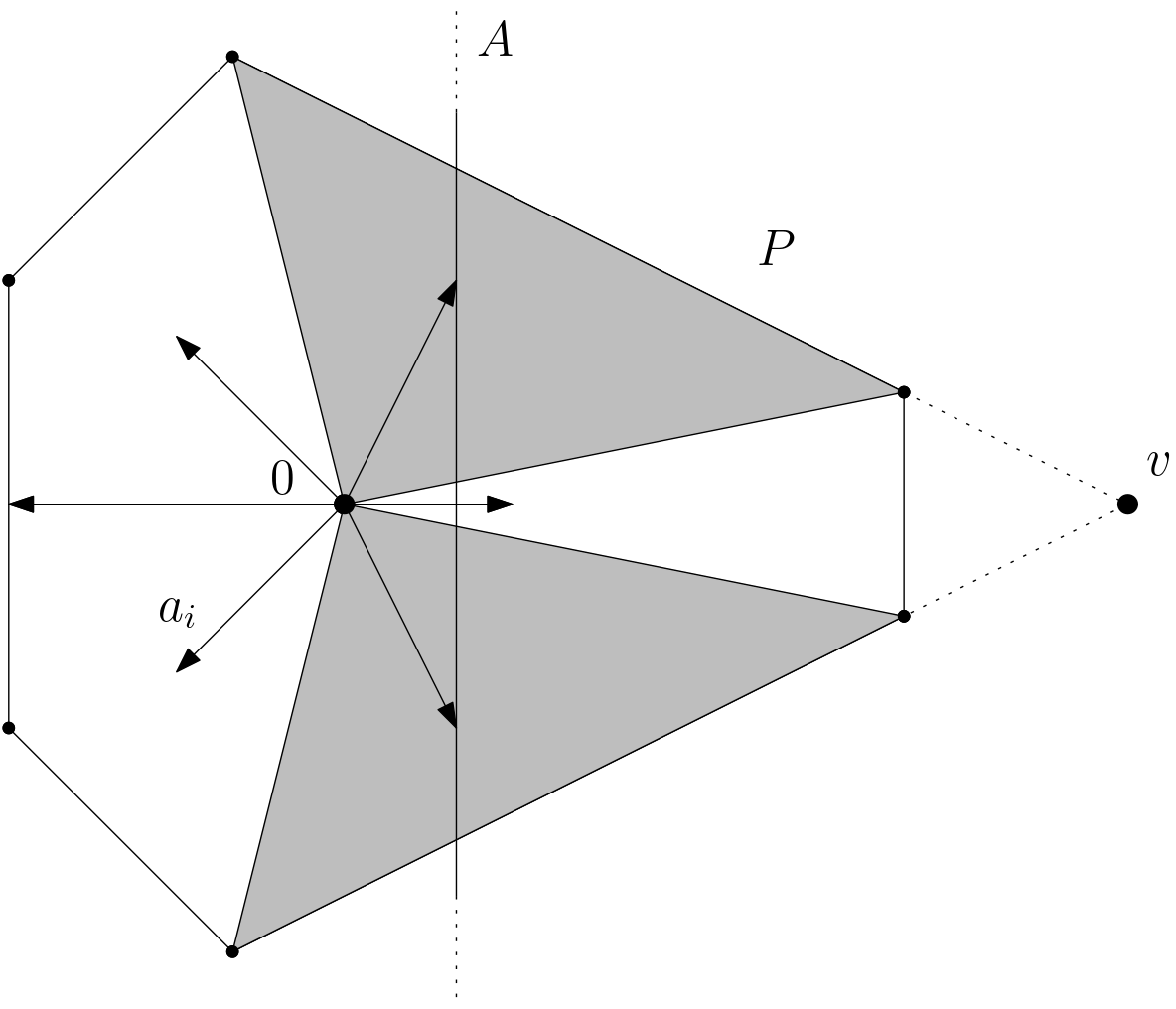}
\caption{The affine situation in Theorem \ref{thm:affine_scc}. Note that in general a subset of the $a_i$'s affinely spans a $k$-subspace, if and only if the affine hulls of the corresponding facets intersect in an $(n-1-k)$-subspace. Here, this is the single point $v$.}
\label{fig:affine}
\end{figure}
\end{theorem}

Unlike Theorem \ref{thm:wu}, which covers the case of reflexive smooth polytopes,
our proof does not give us insight into the characterization of the
equality case. We are, however, able to treat the equality case in two special cases:

%
%


\begin{theorem}
	\label{thm:pyramids}
	Let $P = \{x\in\R^n : \langle x,a_i\rangle \leq 1,\,1\leq i\leq m\}$ be a centered polytope.
	\begin{enumerate}
		\item If $A=\{a_i\}$ for some $1\leq i\leq m$, then equality holds
		in \eqref{eq:affine_scc} if and only if $P$ is a pyramid with
		base $F_i$. 
		\item If $A$ is the hyperplane spanned by the $a_i$'s
		corresponding to all the facets containing a vertex $v$ of $P$, then equality holds in \eqref{eq:affine_scc} if and only if $P$ is a pyramid with apex $v$.
	\end{enumerate}
\end{theorem}
As a byproduct of the proof of Theorem \ref{thm:pyramids}, we will see
alternative proofs of \eqref{eq:affine_scc} in these special cases.  
The first case of Theorem \ref{thm:pyramids} slightly generalizes a
former result by Zhou and He \cite[Thm.~1.2]{zh17}. There an additional technical assumption on
$P$ is made. 
We also point out that for simple polytopes Theorem
\ref{thm:pyramids} implies the following corollary:
\begin{corollary} 
\label{cor:simple}
	Let $P = \{x\in\R^n : \langle x,a_i\rangle \leq
	1,\,1\leq i\leq m\}$ be a centered simple polytope. Let
	$A\subset\R^n$ be an affine subspace spanned by $a_i$'s
	corresponding to all the facets containing a $k$-face of $P$ with
	$1\leq k\leq n-1$.  Then we have equality in
	\eqref{eq:affine_scc} if and only if $P$ is a centered simplex.
\end{corollary}   

In contrast to the description of the equality case in Theorem \ref{thm:wu}, the descriptions of the equality cases in Theorem \ref{thm:pyramids} do not explicitly refer to the normal vectors $a_i$. The following proposition gives an equivalent formulation of the equality case in Theorem \ref{thm:wu}; it shows that the two conditions in Theorem \ref{thm:pyramids} are indeed special cases of the general description in terms of the $a_i$'s: 

\begin{proposition}\label{prop:eq_case} Let $P = \{x\in\R^n : \langle x,a_i\rangle \leq
	1,\,1\leq i\leq m\}$. Then there exist a proper affine subspace $A$ and a complementary  affine
	subspace $A'$ such that $\{a_i:1\leq i \leq m\}\subseteq A
        \cup A'$ 
	if and only if $P$ can be written as
	\begin{equation*} 
	P=\conv(Q_1\cup Q_2),
	\end{equation*} 
	where $Q_1, Q_2\subset\R^n $ are  polytopes with $\dim Q_1+\dim
	Q_2=n-1$  and 
	$\aff Q_1\cap \aff Q_2=\emptyset$.
	\label{prop:geom_affine}
\end{proposition}
Proposition \ref{prop:eq_case} appears to be well-known, but since we are not aware of a proof in the literature, we provide one in Section
\ref{sec:case_of_eq}. We note that the corresponding equality
statement in the case of  linear subspaces (see Theorem \ref{thm:henklinke})   gives
$P=Q_1+Q_2$ where $Q_1, Q_2\subset\R^n$ are polytopes with $\dim
Q_1+\dim Q_2=n$ and $\lin Q_1\cap \lin Q_2=\{0\}$ (see, e.g.,
\cite[Sect.~3]{hl14}). Here $\lin S$ denotes the linear hull of a subset $S\subseteq\R^n$.    

The rest of the paper is organized as follows: Section \ref{sec:basics} contains some preliminaries. In Section \ref{sec:pfs}, we prove Theorem \ref{thm:affine_scc} and Theorem \ref{thm:pyramids}. We give two proofs for Theorem \ref{thm:pyramids} ii), one geometric and one analytic. Finally, in Section \ref{sec:case_of_eq}, we discuss the geometric meaning of the equality case in Theorem \ref{thm:wu} and prove Proposition \ref{prop:eq_case}. 

\section{Preliminaries}
\label{sec:basics}
In this section we give a brief overview of the concepts that are
necessary for the understanding of the paper. We refer to \cite{zie07}
for a detailed introduction into the theory of polytopes and their
face structure, and to \cite{ArtsteinAvidanGiannopoulosMilman2015,Gardner2006,Gruber2007,sch13} for exhaustive background information on Convex Geometry.

\subsection{Polytopes}\label{subsec:polytopes}
A \emph{polytope} $P\subset\R^n$ is, by definition, the convex hull of a finite set $X\subset\R^n$. By the Minkowski-Weyl theorem, $P$ may be represented as 
\[
P = \{ x\in\R^n\colon \langle a_i,x\rangle\leq b_i,~1\leq i \leq m\},
\]
for certain $a_1,\dots,a_m\in\R^n$ and
$b_1,\dots,b_m\in\R$. Conversely, the right-hand side in the above
equation defines a polytope, whenever the set is bounded. We say that
this description is \emph{irredundant} if none of the constraints
$\langle a_i,x\rangle\leq b_i$ may be omitted without changing the
polytope. In this case, the set $F_i = P\cap\{x\in\R^n :\langle
x,a_i\rangle = b_i\}$ is an $(n-1)$-dimensional polytope in the boundary of $P$ and it is called a \emph{facet} of $P$. More generally, a convex subset $F\subseteq P$ with the property that $\lambda x +(1-\lambda)y\in F$, for some $\lambda\in (0,1)$ and $x,y\in P$, implies $x,y\in F$ is called a \emph{face} of $P$. This is equivalent to the existence of a hyperplane $H$ such that $F=P\cap H$ and $P$ is contained in one of the closed half spaces defined by $H$.

  For $P\in\Pn_o$ with an irredundant description
\begin{equation}
\label{eq:irredundant_descr}
P = \{ x\in\R^n : \langle x,a_i\rangle\leq 1,~1\leq i\leq m\},
\end{equation}  
    one defines the \emph{polar} polytope of $P$ as
\[
P^\star = \{y\in\R^n: \langle x,y\rangle\leq 1,~\forall x\in P\} = \conv\{a_1,\dots,a_m\}.
\]
While the inequality description is certainly redundant, the convex hull description is not; the $a_i$'s are precisely the vertices of $P$. 

There is an even stronger duality between the faces of $P$ and $P^\star$. For a face $F\subset P$ of dimension $d\in\{0,\dots,n-1\}$ one defines its polar face as \[
F^\diamond = \{ y\in P^\star : \langle x,y\rangle = 1,~\forall x\in F\}. 
\]
It turns out that $F^\diamond$ is indeed an $(n-d-1)$-face of $P^\star$ and that any $(n-d-1)$-faces of $P^\star$ arises this way. Moreover, we have $(F^\diamond)^\diamond = F$ and $F^\diamond\supseteq G^\diamond$ for $F\subseteq G$.

\subsection{Volume and Centroids}
For a $k$-dimensional polytope $P\subset\R^n$ we denote by $\vol_k(P)$
its $k$-dimensional volume within its affine hull. Likewise, we define
$\centroid(P) = \vol_k(P)^{-1}\int_P x\,\mathrm{d}^k x$ as its centroid. We will need the following formula in order to compute the centroid of a pyramid.

\begin{lemma}
\label{lemma:pyr_formula}
Let $F\subset\R^n$ be an $(n-1)$-dimensional polytope and $v\notin\aff F$. Then,
\begin{equation}
    \label{eq:pyramidcenter}
    \centroid \big(\conv(F\cup\{v\})\big) = \frac{n}{n+1}\centroid(F) + \frac{1}{n+1}v.
\end{equation}
\end{lemma}

\begin{proof}
As $\centroid(\cdot)$ is affinely equivariant, it is enough to
consider the case where $F\subseteq\{x\in\R^n: x_n=0\}$, $\centroid(F)=0$ and $v=e_n$, where
$e_n$ denotes the $n$-th standard unit vector. Let $H_t = \{x\in\R^n \colon x_n=t\}$. Using Fubini's theorem, we have
\[
\centroid(P) = \frac{1}{\vol(P)}\int_0^1 \vol_{n-1}(P\cap H_t)\,\centroid(P\cap H_t)\,\mathrm dt.
\]
In our setting, we have $P\cap H_t = (1-t)F + t e_n$. Thus, it follows that
\[
\begin{split}
\centroid(P) &=  \frac{\vol_{n-1}(F)}{\vol(P)} \Big(\int_0^1 (1-t)^{n-1} t\,\mathrm{d}t\Big)\,e_n\\
&= \frac{\vol_{n-1}(F)}{\vol(P)} \frac{1}{n(n+1)}e_n = \frac{1}{n+1}e_n,
\end{split}
\]
where the last equality follows from the fact that $P$ is a pyramid with height one over $F$ and therefore $\vol(P) = \vol_{n-1}(F)/n$. Given our assumptions, the proof of the Lemma is finished.
\end{proof}

Moreover, we are going to make use of the following additivity
property of the centroid. Consider a finite family of convex bodies
$K_1,\dots,K_m\subseteq\R^n$ whose union $K=K_1\cup\dots\cup K_m$ is
again a convex body and suppose that $K_i\cap K_j$ is a set of
Lebesgue measure zero for all $i\neq j$. Then we have
\begin{equation}
\label{eq:additivity}
\centroid (K) = \frac{1}{\vol (K)}\big(\vol(K_1)\centroid(K_1) + \cdots +\vol(K_m)\centroid(K_m)\big).
\end{equation}

Finally, the following lemma will be used in the proof of Theorem \ref{thm:pyramids}. Here and in the following, for $u \in \R^n$, $u^\bot$ denotes the orthogonal complement of $\lin\{u\}$.

\begin{lemma}\label{lemma:pyr_affine}
	Let $P\subset \R^n$ be an $n$-dimensional polytope, $u \in
        \sph^{n-1}$ and let $f\colon \R\rightarrow \R$, given by
        \begin{equation*} f(t)= \vol_{n-1}((tu+u^\bot) \cap
          P)^{\frac{1}{n-1}}.
        \end{equation*}   
          Let $[\alpha,\beta]=\supp(f)$. If $f$ is affine on $[\alpha,\beta]$ and $f(\beta)=0$, then $P$ is a pyramid with base $(\alpha u+u^\bot) \cap P$ and apex $(\beta u+u^\bot) \cap P$.
\end{lemma}

\begin{proof}
	Let $S = (\alpha u+u^\bot) \cap P$  and $T= (\beta u+u^\bot) \cap P$. Since $f$ is affine, $f(\beta)=0$ and $\vol(P)=\int_\R f(x)^{n-1} dx>0$, we know that $\vol_{n-1}(S)=f(\alpha)>0$. Let $\lambda \in [0,1]$. By the convexity of $P$, we have
	\[\lambda T+ (1-\lambda) S \subseteq \big([\lambda \beta+(1-\lambda)\alpha]u+u^\bot\big) \cap P\eqqcolon P_\lambda.\]
	Combining this with the Brunn-Minkowski inequality \cite[Thm.\ 7.1.1]{sch13}, we obtain
	\begin{equation}\label{eq:pyr_brunn_minkowski}
	\begin{split}f(\lambda \beta+(1-\lambda)\alpha)&\geq\vol_{n-1}(\lambda T+ (1-\lambda) S)^{\frac{1}{n-1}}\\&\geq \lambda\vol_{n-1}(T)^{\frac{1}{n-1}}+(1-\lambda)\vol_{n-1}(S)^{\frac{1}{n-1}}\\&=\lambda f(\beta)+(1-\lambda) f(\alpha).\end{split}
	\end{equation}
	Since $f$ is affine, both inequalities hold with equality. The equality in the Brunn-Minkowski inequality implies that $S$ and $T$ are homothetic (the other equality case being ruled out by the fact that $\vol_{n-1}(S)>0$). Because $\vol_{n-1}(T)=0$, this shows that $T$ is a singleton. Finally, since the polytope $P_\lambda$  contains the polytope $\lambda T+ (1-\lambda) S$, the first equality in \eqref{eq:pyr_brunn_minkowski} implies that $P_\lambda=\lambda T+ (1-\lambda) S$. Since $\lambda \in [0,1]$ was arbitrary, it follows that $P$ is a pyramid with $S$ as its base.
\end{proof}

\section{Proofs of the Theorems}
\label{sec:pfs}

We start with Theorem \ref{thm:affine_scc}.
The basic idea of the proof
is to reduce the problem to the linear case by
replacing $P$ by a certain pyramid $\pyr(P)$ one dimension higher. Rather than performing this replacement step once, we do it recursively, leading to an infinite sequence of pyramids $\pyr(P)$,  $\pyr(\pyr(P))$,  etc. The crucial observation is that the reduction to the linear case becomes stronger in higher dimensions, with the desired estimate as the limiting case.

To this end we define  for $k$-dimensional polytope 
$Q\subset\R^k$ the pyramid $\pyr(Q)$ by 
\[\pyr(Q) =
  \conv((Q\times\{1\})\cup\{-(k+1)e_{k+1}\})\subset\R^{k+1}.\]  
We will need the following properties of this embedding. 
\begin{lemma}
\label{lemma:centered_pyramids}
Let $P\in\Pn_o$ be given as in Theorem \ref{thm:affine_scc}, let $P^{(1)} = \pyr(P)$. Then the following holds: 
\begin{enumerate}
\item \begin{equation*}
         P^{(1)} = \left\{ x\in\R^{n+1} :\left\langle
         \binom{\frac{n+2}{n+1}a_i}{-\frac{1}{n+1}},x
       \right\rangle\leq 1,\,1\leq i\leq m,  x_{n+1}\leq 1 \right\},
\end{equation*} 
\item \begin{equation*}
    \vol_{n+1}(P^{(1)}) = \frac{n+2}{n+1}\vol_n(P),
  \end{equation*}
\item $P^{(1)}$ is centered, i.e.,~$\centroid(P^{(1)})=0$, 
\item Let $C_i^{(1)}$ be the cone given by the facet of $P^{(1)}$
  corresponding to the outer normal vector $\left(\frac{n+2}{n+1}a_i, -\frac{1}{n+1}\right)^T$ and the origin. Then for $1\leq
  i\leq m$   
  \begin{equation*}
    \vol_{n+1}(C^{(1)}_i)=\vol_n(C_i).
\end{equation*} 
\end{enumerate}
\end{lemma}

\begin{proof}
i) and ii) follow directly from the fact that $P^{(1)}$ is indeed a
pyramid;  iii) is a consequence of \eqref{eq:pyramidcenter}.
For iv), let  $\overline{C}_i = C_i\times\{1\}$, $G_1 =
\conv(\overline{C}_i\cup\{-(n+1)e_{n+1}\})$ and
$G_2=\conv(\overline{C}_i\cup\{0\})\subseteq G_1$. Then we have   $C_i^{(1)} = G_1\setminus G_2$ and therefore \[ \begin{split}
\vol_{n+1}(C_i^{(1)}) &= \vol_{n+1}(G_1)-\vol_{n+1}(G_2)\\
&= \frac{n+2}{n+1}\vol_n(C_i) - \frac{1}{n+1}\vol_n(C_i) \\ &= \vol_n(C_i).\qedhere
\end{split}\]
\end{proof}

\begin{proof}[Proof of Theorem \ref{thm:affine_scc}] 
Let
  $A\subseteq\R^n$ be a proper affine space, $d=\dim A$, $I=\{i\in [m]
  :a_i\in A\}$ and  we may assume $\dim\{a_i:i\in I\} =d$.
  For any $k$, let \[\varphi_k:\R^k\to\R^{k+1},~x\mapsto \begin{pmatrix}
	\frac{k+2}{k+1}x\\
	-\frac{1}{k+1} 
	\end{pmatrix}.
	\]
	For $j\geq 1$ and $i\in [m]$ we set \[ a^{(j)}_i = (\varphi_{n+j-1} \circ \cdots\circ\varphi_{n})(a_i)
	\in \R^{n+j},\]
and let  \(L^{(j)} = \lin\{a_i^{(j)}:i\in
I\}\subseteq\R^{n+j}\). Observe that the vectors $a_i^{(j)}$ have the form
\[
a_i^{(j)}
	= \begin{pmatrix}
	\frac{n+j+1}{n+1}a_i\\
	c_{n+1}\\
	\vdots \\
	c_{n+j}
	\end{pmatrix},
\]
where
\[
c_{n+k}=-\frac{n+j+1}{(n+k)(n+k+1)},\quad 1\leq k\leq j.
\]
The $c_{n+k}$'s only depend on $n$ and $j$, but not on $a_i$.
Therefore, $L^{(1)}$ is a $(d+1)$-dimensional linear
space and since the matrix $(a_i^{(j+1)}: i\in I )$ differs
from $(a_i^{(j)}:i\in I)$ only by an additional constant row
and a multiplication of the first $n+j-1$ rows, we have $\dim L^{(j)}=d+1$
for all $j\geq 1$.

Consider the pyramids $P^{(j)}=\pyr(P^{(j-1)})$ with $P^{(0)}=P$. A
repeated application of Lemma \ref{lemma:centered_pyramids} i) and iii) shows that
each $P^{(j)}$ is a centered pyramid that has the vectors 
$\{a_i^{(j)}:i\in [m]\}$ among its normal vectors, and from  Lemma
\ref{lemma:centered_pyramids} ii)  we get
\begin{equation*}
     \vol_{n+j}(P^{(j)}) = \left(\prod_{k=1}^j\frac{n+k+1}{n+k }\right) \vol_n(P) = \frac{n+j+1}{n+1}\voln(P).
  \label{eq:volume}
\end{equation*}
Let $C_i^{(j)}$ be the cone of $P^{(j)}$ corresponding to
$a_i^{(j)}$. Lemma \ref{lemma:centered_pyramids} iv) shows that  
$\vol_{n+j}(C_i^{(j)})=\vol_n(C_i)$,    
	and so by Theorem \ref{thm:henklinke} applied to $P^{(j)}$ and $L^{(j)}$ we obtain
	\[\begin{split} \sum_{i\in I}\voln(C_i)& = \sum_{i\in I} \vol_{n+j}(C_i^{(j)})\\
	&\leq \frac{d+1}{n+j} \vol_{n+j}(P^{(j)}) \\ &= \frac{\dim A+1}{n+1}\frac{n+j+1}{n+j} \voln(P).
	\end{split}\] The claim follows from letting $j\to\infty$.  
\end{proof}

Before we come to the proofs of Theorem \ref{thm:pyramids} i) and ii), we observe that equality holds in \eqref{eq:affine_scc}, whenever $\{a_i : 1\leq i \leq m\}\subseteq A\cup A'$, where $A'$ is complementary to $A$; To see this, it suffices to apply \eqref{eq:affine_scc} to both $A$ and $A'$ and obtain
\[
\vol(P) = \sum_{i:a_i\in A}\vol(C_i) + \sum_{i:a_i\in A'} \vol(C_i) \leq \vol(P).
\]
Thus, we have equality in \eqref{eq:affine_scc} for $A$ (and also for $A'$). In view of Proposition \ref{prop:eq_case}, we thus only need to show the ``only if'' parts for the equality cases in Theorem \ref{thm:pyramids}.

We start with case of $A$ being a singleton. Our proof is inspired by the proof of Gr\"unbaum's theorem on central sections of centered convex bodies \cite{gru60}.
\begin{proof}[Proof of Theorem \ref{thm:pyramids} i)]
Without loss of generality,
we assume that $F_i= P \cap \{x \in \R^n : \langle
e_1,x\rangle=-\alpha\}$ for an
appropriately chosen $\alpha>0$. Let $Q = \conv(F_i \cup \{\beta e_1\})$, where $\beta >-\alpha$ is chosen such that $\vol(Q)=\vol(P)$. We define two functions $\R \rightarrow \R$ via
\[f(t) = \vol_{n-1}((te_1+e_1^\bot) \cap P)^{\frac{1}{n-1}}, \quad g(t) = \vol_{n-1}((te_1+e_1^\bot) \cap Q)^{\frac{1}{n-1}}.\] 
If $\langle e_1, c(Q) \rangle \geq \langle e_1, c(P) \rangle$, then by
Lemma \ref{lemma:pyr_formula} it would follow that
\[\vol(C_i)\leq \vol(\conv(F_i \cup \{\centroid(Q)\}))=\frac{1}{n+1} \vol(Q),\]
as desired. Recalling that $P$ is centered, we have to show for $\gamma =\langle e_1, \centroid(Q) \rangle$  that \[\gamma = \langle e_1, \centroid(Q)-\centroid(P) \rangle = \int_{-\infty}^{\infty} t [g(t)^{n-1}-f(t)^{n-1}]\mathrm d t \geq 0,\] with equality if and only if $P$ is a pyramid.

Since $Q$ is a pyramid with base orthogonal to $e_1$, $g$ is affine on
$\supp(g)=[-\alpha,\beta]$. By Brunn's concavity principle \cite[Thm.\ 1.2.1]{ArtsteinAvidanGiannopoulosMilman2015}, $f$ is concave on $\supp(f)$. Hence, $g-f$ is convex on $\supp(f) \cap \supp(g)$. In fact, we have $\supp(f) \subseteq \supp(g)$: If there was a $t>\beta$ with $f(t)>0$, then the concavity of $f$ would imply $f>g$ on $\supp(g)$, in contradiction to $\vol(Q)=\vol(P)$. Hence, $g-f$ is convex on $\supp(f)$ and the sublevel set \[\supp(f)\cap\{g-f \leq 0\}=\supp(f)\cap\{g^{n-1} -f^{n-1}\leq 0 \}\] is convex. Since $f(-\alpha) =g(-\alpha)$, it follows that $\supp(f)\cap\{g-f \leq 0\}=[-\alpha,\tau]$ for a $\tau\leq \beta$. On $[\tau,\beta]$ we have $g \geq f$, leading to the desired estimate
\[\begin{split} \gamma &= \int_{-\alpha}^\tau t [g(t)^{n-1}-f(t)^{n-1}]\mathrm d t+\int_{\tau}^\beta t [g(t)^{n-1}-f(t)^{n-1}]\mathrm d t\\
&\geq \int_{-\alpha}^\tau \tau [g(t)^{n-1}-f(t)^{n-1}]\mathrm d t+\int_{\tau}^\beta \tau [g(t)^{n-1}-f(t)^{n-1}]\mathrm d t\\
&=\tau\left(\int_{-\alpha}^\beta [g(t)^{n-1}-f(t)^{n-1}]\mathrm d t\right)=\tau\left(\vol(Q)-\vol(P)\right)=0.
\end{split}\]
Equality holds if and only if $g=f$ on $[-\alpha,\beta]$. It is clear
that this is the case if $P$ is a pyramid with base $F_i$; the other direction follows from Lemma \ref{lemma:pyr_affine}.
\end{proof} 

Next, we give two proofs of Theorem \ref{thm:pyramids} ii), corresponding to two different perspectives on the problem. The first proof has a more geometric flavor, whereas the second proof is of a probabilistic nature. 
\begin{proof}[Geometric proof of Theorem \ref{thm:pyramids} ii)]
 Let $I\subseteq [m]$ be the set of indices such that $\langle v,a_i\rangle =1$, i.e., $A=\aff\{a_i:i\in I\}$. Since $P$ is centered, we have $-\frac 1 n v \in P$ (see \cite[Sect.~34]{bf34}). For $i\in I$, we consider the cones $\overline{C}_i=\conv(F_i\cup\{-(1/n) v\})\subseteq P$, where $F_i$ is the facet of $P$ with normal $a_i$. By the volume formula for pyramids, we have $\vol(\overline{C}_i) = \frac{n+1}{n}\vol(C_i)$. As the $\overline{C}_i$'s intersect in a set of measure zero, we obtain
\begin{equation}
\label{eq:overlinec}
\vol(P)\geq \sum_{i\in I}\vol(\overline{C}_i) = \frac{n+1}{n}\sum_{i\in I}\vol(C_i).
\end{equation}
So we have reproven Theorem \ref{thm:affine_scc} in this case. In order to have equality in the above, we must have $P=\bigcup_{i\in I}\overline{C}_i$. Let $J=[m]\setminus I$. Then we have 
\begin{equation}
\label{eq:subdivision}
\langle -(1/n)v,a_j\rangle=1,~\forall j\in J,
\end{equation}
since otherwise, the cone $C_j$ would have a positive volume and we could not achieve equality in \eqref{eq:overlinec}.

For $j\in J$, let $Q_j=\conv(F_j\cup\{v\})\subseteq P$. Just like the
$\overline{C}_i$'s, the $Q_j$'s subdivide $P$, i.e., $P=\bigcup_{j\in
  J}Q_j$ and the pyramids intersect in sets of measure zero. By
\eqref{eq:pyramidcenter}, we have $\centroid(Q_j) =
\frac{n}{n+1}\centroid(F_j) + \frac 1 n v$ and in view of
\eqref{eq:additivity} we may write 
\[ 0 = \centroid(P) = \sum_{j\in J} \frac{\vol(Q_j)}{\vol(P)} \left(  \frac{n}{n+1}\centroid(F_j) + \frac{1}{n+1} v \right).\]
Multiplying with $(n+1)/n$ and rearranging yields 
\[-\frac 1 n v = \sum_{j\in J} \frac{\vol(Q_j)}{\vol(P)}\left(-\frac 1 n v\right) = \sum_{j\in J} \frac{\vol(Q_j)}{\vol(P)} \centroid(F_j).\]
Hence, \eqref{eq:subdivision} gives for any $j\in J$ 
\[1 = \langle -\frac 1 n v, a_j\rangle = \sum_{k\in J} \frac
  {\vol(Q_k)}{\vol(P)} \langle \centroid(F_k), a_j\rangle.\] 
Towards a contradiction, assume that $J$ contains more than one element. Then there is a $k\in J\setminus \{j\}$. Since $\centroid(F_k)\in\mathrm{relint} F_k$ we have $\langle \centroid(F_k), a_j\rangle < 1$. It follows that $1 < \sum_{k\in J}\vol(Q_k)/\vol(P) = 1 $. Therefore, $J$ can contain only one element, which corresponds to the case that $P$ is a pyramid with apex $v$.
\end{proof}

We now come to the second proof of Theorem \ref{thm:pyramids} ii) via a probabilistic approach.

\begin{proof}[Analytic proof of Theorem \ref{thm:pyramids} ii)]
Again, we only show the ``only if'' part of the equality case. To this end, we assume that $\vol(P)=1$, which is not a restriction as both sides of \eqref{eq:affine_scc} are $n$-homogeneous. By definition, we have $\centroid(P)=\E[X]$, 
where $X$ is a uniformly distributed random vector in $P$. We consider the functional 
\[ f:\R^n\to\R,~f(x) =\frac 1 n \sum_{i:a_i\in A} \mathrm{dist}(x,\aff F_i)\,\vol_{n-1}(F_i),\]
 where $\mathrm{dist}(x,\aff F_i)$ is the signed Euclidean distance to $\aff F_i$, oriented such that it is non-negative inside $P$. Note that for $x\in P$ one has \[ f(x) =  \sum_{i:a_i\in A} \vol\big(\conv(F_i\cup\{x\})\big).\] As $f$ is an affine map, we have
\begin{equation}
\label{eq:expectation}
\sum_{i:a_i\in A} \vol(C_i) = \E[f(X)] = \int_0^1 \Prob_X(f\geq t)\mathrm d t = 1-\int_0^1\Prob_X(f<t)\mathrm d t.
\end{equation}
We consider the function $p:[0,1]\to [0,1],~t\mapsto\Prob_X(f<t)^{\frac 1 n}$.
We have $p(0)=0$ and $p(t )=1$, for $t\geq m=\max f(P)\leq 1$. Let $H(t) = \{x\in\R^n:f(x)\leq t\}$ be the half-space where $f\leq t$.
Since the vertex $v$ is the unique point that is contained in all facets $F_i$, where $a_i\in A$, we have $0\in f(P)$ and $f(x)=0$ for $x\in P$, if and only if $x=v$. Thus, $P\cap H(0)=\{v\}$. Using the inclusion
\begin{equation}
\label{eq:inclusion}
P\cap H(t)\supseteq\frac t m \big(P\cap H(m)\big) + \frac{m-t}{m}v,
\end{equation} 
we deduce that, for any $t\in [0,m]$,
\begin{equation}
\label{eq:fradelizi} 
\begin{split}
p(t)& = \vol(P\cap H(t))^{\frac 1 n} \geq \vol\Big(\frac t m \big(P\cap H(m)\big) +\frac{m-t}{m}v\Big)^\frac{1}{n}\\
& = \frac t m \vol \big(P\cap H(m)\big)^{\frac 1 n} = \frac t mp(m)=\frac t m.
\end{split}
\end{equation}

Applying this to \eqref{eq:expectation}, we have 
\[\begin{split}
\sum_{i:a_i\in A} \vol(C_i) &= 1-\int_0^m p(t)^n\mathrm d t -(1-m)\\
&\leq m-\int_0^m \left(\frac t m\right)^n\mathrm d t= \frac{mn}{n+1}\leq \frac{n}{n+1}.
\end{split}\]
By our assumption that $\vol(P)=1$, this is \eqref{eq:affine_scc}. In order to have equality, we must have $m=1$ and equality in \eqref{eq:fradelizi}, i.e., $\vol(P\cap H(t))=t^n$ for $t\in [0,1]$. This is equivalent to $\vol_{n-1}(P\cap \{x\in\R^n:f(x)= t\})=nt^{n-1}$ for $t\in [0,1]$. By Lemma \ref{lemma:pyr_affine}, this implies that $P$ is a pyramid with apex $v$.
\end{proof}

\begin{remark}
It is natural to ask whether the assumption in Theorem \ref{thm:pyramids} ii) that $v$ is a vertex of $P$ can be removed. In other words, is it possible to adapt our proofs to the situation where the hyperplanes $\{x \in \R^n : \langle a_i,x\rangle=1\}$, $a_i \in A$, intersect in a single point $v$ that is not necessarily contained in $P$ (cf.\ Figure \ref{fig:affine})? Both proofs of Theorem \ref{thm:pyramids} ii) make use of the assumption that $v \in P$: In the first proof, we use it to derive $-\frac{1}{n}v \in P$; in the second proof, it ensures that $p$ is concave on $[0,\max f(P)]$. It is not clear how the first proof could be modified to dispense with the assumption. In the second proof, a suitable upper bound on $\max f(P)$ in terms of $\min f(P)$ would be sufficient: The concavity of $p$ on $[\min f(P),\max f(P)]$ leads to the desired estimate if we additionally assume that $\max f(P) \leq 1- \frac{\min f(P)}{n}$.
\end{remark}

\section{Complementary Affine Subspaces}\label{sec:case_of_eq}

To conclude, let us have a closer look at the characterization of the
equality case as it has been formulated by K.-Y. Wu: A smooth and reflexive polytope $P = \{x\in\R^n : \langle x, a_i\rangle\leq 1,~1\leq i \leq m\}$ satisfies the affine subspace concentration condition \eqref{eq:affine_scc} for an affine $d$-subspace $A$ with equality if and only if the normal vectors $\{a_i:1\leq i\leq m\}$ of $P$ are contained in $A\cup A'$, where $A'$ is an affine $(n-d-1)$-subspace complementary to $A$.

At first glance, this condition may appear rather technical, but in
fact, it has a strong geometric interpretation for the polytope $P$:
Since the $a_i$'s are the vertices of $P^\star$, the condition that
$\{a_1,\dots,a_m\}$ is contained in $A\cup A'$ is equivalent to
$P^\star = \conv( P_1\cup P_2)$, where $P_1 $ is a $d$-polytope and
$P_2$ is an $(n-d-1)$-polytope and $\aff P_1=A$ and $\aff P_2=A'$
are complementary affine spaces. In general, 
a polytope that can be expressed as the convex hull of two polytopes
$Q_1$ and $Q_2$ in complementary affine subspaces is also called the
\emph{join} of $Q_1$ and $Q_2$ \cite[p.~390]{hrgz17}.
Therefore, the statement of Proposition 1.4 can be reformulated as 
that an $n$-dimensional polytope $P$ is the join of a $d$-polytope
$Q_1$and a $(n-d-1)$-polytope $Q_2$ if and only $P^*$ is. Since we
could not find a reference for this certainly well-known fact we add a
proof.

\begin{proof}[Proof of Proposition \ref{prop:geom_affine}]
 By polarity, it is 
 enough
to prove that $P$ being the join of $Q_1$ and $Q_2$ implies that $P^\star$ is the join of two polytopes $P_1$ and $P_2$ of appropriate dimension.

So let $P=\conv(Q_1\cup Q_2)$ with $Q_1$ and $Q_2$ as in the statement of the proposition. First, we show that $Q_1$ and $Q_2$ are faces of $P$. Let $x_1 \in Q_1$, $x_2 \in Q_2$ and $L=\lin((Q_1-x_1) \cup (Q_2-x_2))$. Since $\dim Q_1 +\dim Q_2 =n-1$, we have $\dim L \leq n-1$. Choosing a vector $u \in L^\bot \setminus \{0\}$, the linear functional $f \colon \R^n \rightarrow \R$, $x\mapsto \langle u,x\rangle$ satisfies $f(Q_1)=\{\alpha\}$ and $f(Q_2)=\{\beta\}$ for certain $\alpha,\beta \in \R$. 
Since $P$ is $n$-dimensional and of the form $P= \conv(Q_1\cup Q_2)$, we have $\alpha \neq \beta$, $f(P)=\conv\{\alpha,\beta\}$ and
\[f^{-1}(\{\alpha\}) \cap P = Q_1, \quad f^{-1}(\{\beta\}) \cap P = Q_2. \]
This shows that $Q_1$ and $Q_2$ are faces of $P$.


The notion of a \textit{polar face} was introduced in Section \ref{subsec:polytopes}. We consider the polar faces $P_i=Q_i^\diamond\subseteq
P^\star$, $i\in\{1,2\}$, of the two faces
$Q_1,Q_2 \subset P$. Note that $\dim P_1 = n-d-1$ and $\dim P_2 = d$. Clearly, we
have $\conv(P_1\cup P_2)\subseteq P^\star$. If the inclusion was
strict, we find a vertex $v$ of $P^\star$ which is neither a vertex of
$P_1$, nor of $P_2$. Consider the  corresponding facet $F=v^\diamond$
of $P$. Since $v$ is not contained in $P_1\cup P_2$, it follows by
polarity that neither $Q_1$, nor $Q_2$ is contained in $F$. But $F_i=Q_i\cap F$ is a face of $Q_i$. Thus, we have $\dim F_1 \leq d-1$ and $\dim F_2 \leq n-d-2$. Due to the assumption $P=\conv(Q_1\cup Q_2)$, the vertices of $F$ are contained in $Q_1\cup Q_2$, i.e., $F = \conv(F_1\cup F_2)$. It follows that \[\dim F \leq 1 +\dim F_1 +\dim F_2 = 1+(d-1)+(n-d-2)= n-2,\]  a contradiction. So we have proven $P^\star = \conv(P_1\cup P_2)$. Since $\dim(P^\star) = n$ and $P^\star\subseteq \aff(P_1\cup P_2)$, we have $\aff(P_1\cup P_2)=\R^n$, so the affine hulls of $P_1$ and $P_2$ are indeed complementary.
\end{proof}

We recall that an $n$-polytope $P$ is called \emph{simple} if every
vertex $v$ of $P$ is contained in exactly $n$ edges, or, equivalently,
in exactly $n$ facets of $P$. For a simple polytope $P=\{x\in\R^n :
\langle x,a_i\rangle\leq  1,~1\leq i\leq m\}$ the property
$\{a_1,\dots,a_m\}\subseteq A\cup A'$, for some pair of complementary
proper affine subspaces of $\R^n$, is equivalent to the fact that $P$
is a simplex. Indeed, we obtain from this that $P^\star =
\conv(P_1\cup P_2)$, where the affine hull of $P_1$ is $A$ and the
affine hull of $P_2$ is $A'$. Since $P$ is simple, $P^\star$ is
simplicial, i.e., all faces of $P^\star$ are simplices. As we saw in
the proof of Proposition \ref{prop:geom_affine}, the polytopes $P_1$
and $P_2$ are faces of $P^\star$, so they are simplices of dimension $\dim A$ and $n-1-\dim A$, respectively. Hence, $P^\star$ is a simplex, which implies that $P$ is a simplex as well.

As smooth polytopes are simple by definition, we see that simplices are the only equality cases in Theorem \ref{thm:wu}.

We conclude by providing a proof of Corollary \ref{cor:simple}.

\begin{proof}[Proof of Corollary \ref{cor:simple}]
Let $F$ be a $k$-face of $P$. Since $P$ is simple, there are exactly
$n-k-1$ vectors among the $a_i$'s that satisfy $F\subseteq
F_i$. Without loss of generality, we assume that $a_1,...,a_{n-k-1}$ are these vectors. In view of Theorem \ref{thm:pyramids} i), we obtain
\begin{equation}
\label{eq:1dimineq}
\vol(C_i)\leq\frac{1}{n+1}\,\vol(P),\text{ for all }1\leq i \leq n-k-1.
\end{equation} 
Summing up these inequalities gives \eqref{eq:affine_scc} for $P$ and $A$ where equality holds, if and only if equality holds in each of the inequalities in \eqref{eq:1dimineq}. In particular, equality holds, only if $P$ is a pyramid with base $F_1$. Since $P$ is simple, this implies that $P$ is a simplex.
\end{proof} 







\end{document}